\documentclass[a4paper,11pt]{article}
\usepackage[margin=1.0in]{geometry}

\usepackage{hyperref}
\usepackage{algorithm,algorithmic}
\usepackage{amsmath}
\usepackage{amsthm}
\usepackage{amsfonts}
\usepackage{latexsym}
\usepackage{amssymb}
\usepackage{xspace}
\usepackage{color}
\usepackage{fullpage}
\usepackage{epsfig}
\usepackage{tikz}
\usepackage{array}
\usepackage{comment}

\newtheorem{theorem} {Theorem}
\newtheorem{lemma} {Lemma}
\newtheorem{definition} {Definition}
\newtheorem{corollary} {Corollary}
\newtheorem{observation} {Observation}

\newcommand{\mF}{\mathcal{F}}
\newcommand{\mV}{\mathcal{V}}

\newcommand{\mP}{\mathcal{P}}
\newcommand{\mK}{\mathcal{K}}
\newcommand{\mS}{\mathcal{S}}

\newcommand{\mM}{\mathcal{M}}

\newcommand{\onesvec}{\textbf{1}}

\newcommand{\card}{\textrm{card}}

\newcommand{\reals}{\mathbb{R}}

\title{Linear-memory and Decomposition-invariant \\
Linearly Convergent Conditional Gradient Algorithm\\
 for Structured Polytopes}
\date{}
\author{Dan Garber and Ofer Meshi\\
Toyota Technological Institute at Chicago \\ 
\small{\{dgarber, meshi\}@ttic.edu}
}
\begin{document}

\maketitle

\begin{abstract}
Recently, several works have shown that natural modifications of the classical conditional gradient method (aka Frank-Wolfe algorithm) for constrained convex optimization, provably converge with a linear rate when:
i) the feasible set is a polytope, and
ii) the objective is smooth and strongly-convex. However, all of these results suffer from two significant shortcomings:
\begin{enumerate}
\item large memory requirement due to the need to store an explicit convex decomposition of the current iterate, and as a consequence, large running-time overhead per iteration
\item
the worst case convergence rate depends unfavorably on the dimension
\end{enumerate}

In this work we present a new conditional gradient variant and a corresponding analysis that improves on both of the above shortcomings. In particular:
\begin{enumerate}
\item
both memory and computation overheads are only linear in the dimension
\item
in case the optimal solution is sparse, the new convergence rate replaces a factor which is at least linear in the dimension in previous works, with a linear dependence on the number of non-zeros in the optimal solution
\end{enumerate}

At the heart of our method, and corresponding analysis, is a novel way to compute decomposition-invariant \textit{away-steps}.
While our theoretical guarantees do not apply to any polytope, they apply to several important structured polytopes that capture central concepts such as paths in graphs, perfect matchings in bipartite graphs, marginal distributions that arise in structured prediction tasks, and more. Our theoretical findings are complemented by empirical evidence which shows that our method delivers state-of-the-art performance.
\end{abstract}

\section{Introduction}

The efficient reduction of a constrained convex optimization problem to a constrained linear optimization problem is an appealing algorithmic concept, in particular for large-scale problems. The reason is that for many feasible sets of interest, the problem of minimizing a linear function over the set admits much more efficient methods than its non-linear convex counterpart. Prime examples for this phenomenon include various structured polytopes that arise in combinatorial optimization, such as the path polytope of a graph (aka the unit flow polytope), the perfect matching polytope of a bipartite graph, and the base polyhedron of a matroid, for which we have highly efficient combinatorial algorithms for linear minimization that rely heavily on the specific rich structure of the polytope \cite{schrijver}.
At the same time, minimizing a non-linear convex function over these sets usually requires the use of generic interior point solvers that are oblivious to the specific combinatorial structure of the underlying set, and as a result, are often much less efficient.
Another important example includes structured sets of matrices such as the \textit{spectrahedron}, i.e., convex-hull of unit-trace positive semidefinite matrices, or the  \textit{nuclear ball} that are central to many machine learning problems, such as \textit{matrix completion}, for which linear optimization amounts to computing the leading eigenvector or leading pair of singular vectors, whereas, algorithms for non-linear convex optimization over these sets often rely on very expensive singular value decompositions.
Indeed, it is for this reason, that the conditional gradient (CG) 
method (aka Frank-Wolfe algorithm), a method for constrained convex optimization that is based on solving linear subproblems over the feasible domain, has regained much interest in recent years in the machine learning, signal processing and optimization communities. It has been recently shown that the method delivers state-of-the-art performance on many problems of interest, see for instance \cite{Jaggi10, Jaggi13a, Dudik12a, Dudik12b, Hazan12, ShalevShwartz11, Laue12, Ying12, Hazan16, joulin2014fw}.

As part of the regained interest in the conditional gradient method, there is also a recent effort to understand the convergence rates and associated complexities of conditional gradient-based methods, which is in general far less understood than other first-order methods, e.g., the projected gradient method. It is known, already from the first introduction of the method by Frank and Wolfe in the 1950's \cite{FrankWolfe}, and the somewhat later work of Polyak and Levitin \cite{Polyak}, that the method converges with a rate of roughly $O(1/t)$ for minimizing a smooth convex function over a convex and compact set, which matches the rate of the standard projected gradient method for the same setting.
%However, even under an additional standard assumption such as strong convexity of the objective function, it is not clear if the convergence rate of the method improves, while certain lower bounds, such as in \cite{Lan13,GH13} suggest that such an improvement, even if possible, should come with a worse dependence on the problem's parameters, such as the dimension, which is a phenomena that does not occur for the projected gradient method, for instance.
However, it is not clear if this convergence rate improves under an additional standard strong-convexity assumption.
In fact, certain lower bounds, such as in \cite{Lan13,GH13}, suggest that such improvement, even if possible, should come with a worse dependence on the problem's parameters (e.g., the dimension), which is a phenomena that does not occur for the projected gradient method, for instance.
Nevertheless, over the past years, various works tried to design natural variants of the CG method that converge provably faster under the additional strong convexity assumption, or a slightly weaker assumption, without dramatically increasing the per-iteration complexity, which is the main appeal for these methods.
For instance, Gu{\'{e}}Lat and Marcotte \cite{GueLat1986} showed that
a CG variant which uses the concept of \textit{away-steps} converges exponentially fast
in case the objective function is strongly convex, the feasible set is a polytope, and the optimal solution is located in the interior of the set.
%For instance, Gu{\'{e}}Lat and Marcotte \cite{GueLat1986} showed that in case the objective function is strongly convex and the feasible set is a polytope, then, in case the optimal solution is located in the interior of the set, a CG variant which uses the concept of \textit{away-steps}, which was already proposed in \cite{FrankWolfe}, converges exponentially fast.
A similar result was presented by Beck and Teboulle \cite{BeckTaboule} who considered a specific problem they refer to a \textit{the convex feasibility problem} over an arbitrary convex set. They also obtained a linear convergence rate under the assumption that an optimal solution that is far enough from the boundary of the set exists. In both of these works, the exponent depends on the distance of the optimal solution from the boundary of the set, which in general can be arbitrarily small.
Later, Ahipasaoglu, Sun and Todd \cite{Ahipasaoglu08} showed that in the specific case of minimizing a smooth and strongly convex function over the unit simplex, a variant of the CG method which also uses away-steps, converges with a linear rate. Unfortunately, it is not clear from their analysis how this rate depends on natural parameters of the problem such as the dimension and the condition number of the objective function.

Recently, Garber and Hazan presented the first natural linearly-converging CG variant for polytopes without any restricting assumptions on the location of the optimum. The exponent in their convergence rate depends on various geometric parameters of the polytope \cite{GH13}. It is important to note, that while, in theory, these geometric parameters can result in an arbitrarily bad convergence rate, for polytopes for which it makes sense to apply the CG method, i.e., there exists an highly efficient algorithm to solve the linear subproblems, such as polytopes that arise in combinatorial optimization problems, these parameters are quite reasonable and can be efficiently computed.
In a follow-up work, Lacoste-Julien and Jaggi \cite{Jaggi13c,lacoste2015linear_fw} gave a refined affine-invariant analysis of an algorithm presented in \cite{GueLat1986} which also uses away steps, and showed that it also converges exponentially fast in the same setting as the Garber-Hazan result.
In a later work, Beck and Shtern \cite{Beck15} gave a different, duality-based, analysis for the algorithm of \cite{GueLat1986}, and showed that it can be applied to a wider class of functions than purely strongly convex functions. However, the explicit dependency of their convergence rate on the dimension is suboptimal, compared to \cite{GH13, lacoste2015linear_fw}.
Aside from the polytope case, Garber and Hazan have shown recently that in case the feasible set is strongly-convex and the objective function satisfies certain strong convexity-like proprieties, then the standard CG method converges with an accelerated rate of $O(1/t^2)$ \cite{GH15}.

Despite the exponential improvement in convergence rate in the polytope case obtained in recent results, all of these results suffer from two major drawbacks. First, while in terms of the number of calls per-iteration to the linear optimization oracle, these methods match the standard CG method, i.e., a single call per iteration, the overhead of other operations both in terms of running times and memory requirements is significantly worse. The reason is that in order to apply the so-called away-steps, which all methods use in order to obtain the accelerated rate, they require to maintain at all times an explicit decomposition of the current iterate into vertices of the polytope. Maintaining such a decomposition and computing the away-steps, even with efficient implementations of incremental decomposition procedures, such  as suggested in \cite{Beck15}, require both memory and per-iteration runtime overheads that are at least quadratic in the dimension. This is much worse than the standard CG method, whose memory and runtime overheads are only linear in the dimension. Second, the convergence rate of all previous linearly convergent CG methods depends explicitly on the dimension. While it is known that this dependency is unavoidable in certain cases, e.g., when the optimal solution is, informally speaking, dense (see for instance the lower bound in \cite{GH13}), it is not clear that such an unfavorable dependence is mandatory when the optimum is sparse.

In this paper, we revisit the application of CG variants to smooth and strongly-convex optimization over polytopes. We introduce a new variant which overcomes both of the above shortcomings from which all previous linearly-converging variants suffer. The main novelty of our method, which is the key to its improved performance, is that unlike previous variants, it is decomposition-invariant, i.e., it does not require to maintain an explicit convex decomposition of the current iterate. This principle proves to be crucial both for eliminating the memory and runtime overheads, as well as to obtaining shaper convergence rates for instances that admit a sparse optimal solution. 

We give a detailed comparison of our method to previous art in Table \ref{table:compare}.
We also provide empirical evidence that the proposed method delivers state-of-the-art performance on several tasks of interest.
While our method is less general than previous ones, i.e., our theoretical guarantees do not hold for arbitrary polytopes, they readily apply for many structured polytopes that capture important concepts such as paths in graphs, perfect matchings in bipartite graphs, Markov random fields, and more.
We also specify how to apply the method to arbitrary polytopes, but without giving formal convergence guarantees. 

\begin{table*}\label{table:compare}
\begin{center}
  \begin{tabular}{| l | c | c | c | c |}
    \hline
    Paper &  \#iterations to $\epsilon$ err. & \#LOO calls & runtime & memory \\ \hline
    Frank \& Wolfe \cite{FrankWolfe} & $\frac{\beta{}D^2}{\epsilon}$ & 1& $n$ & $n$  \\ \hline
    Garber \& Hazan \cite{GH13} & $\frac{n\beta{}D^2}{\alpha}\log(1/\epsilon)$ & 1& $n^2$ & $n^2$  \\ \hline
    Lacoste-Julien \& Jaggi   \cite{lacoste2015linear_fw}  & $\frac{n\beta{}D^2}{\alpha}\log(1/\epsilon)$ & 1& $n^2$ & $n^2$   \\ \hline
    Beck \& Shtern \cite{Beck15} & $\frac{n^2\beta{}D^2}{\alpha}\log(1/\epsilon)$ & 1& $n^2$ & $n^2$ \\ \hline
    This paper & $\frac{\card(x^*)\beta{}D^2}{\alpha}\log(1/\epsilon)$ & 2& $n$ & $n$ \\ \hline
  \end{tabular}
  \caption{Comparison with previous works. The third column gives the number of calls to the linear optimization oracle per iteration, fourth column gives the overall additional arithmetic complexity per iteration, and the fifth  column gives the worst case memory requirement of the algorithm. To get lower complexity and memory requirements for the algorithms in \cite{GH13, lacoste2015linear_fw, Beck15}, we assume they all employ an algorithmic version of Carathéodory's theorem to maintain a convex decomposition of the iterate to at most $n+1$ vertices, as fully detailed in \cite{Beck15}. We note that the bound on number of iterations in the analysis of \cite{lacoste2015linear_fw} does not depend explicitly on the dimension $n$, but on the squared inverse \textit{pyramidal width} of $\mP$, which is difficult to evaluate. However, already for the simplest polytope, i.e., the unit simplex, this quantity is proportional to $n$.}
\end{center}
\end{table*}

\subsection{Organization of the paper}
The rest of this paper is organized as follows. In Section \ref{sec:prelim} we give preliminaries and notation, and present the exact setting considered in this paper. In Section \ref{sec:approach} we briefly present the conditional gradient method and its previous away-steps-based variants, and present our new method: a decomposition-invariant pairwise conditional gradient algorithm. In this section we also give our main theorem which details the novel convergence rate of our method. In Section \ref{sec:polyExample} we briefly describe several important polytopes that fall into our assumptions, and detail the application of our method for optimization over these polytopes. In Section \ref{sec:analysis} we give a complete analysis of our method and prove the main theorem. In Section \ref{sec:extension} we detail how to apply our approach to a broader class of polytopes, though we do not complement our algorithm with a convergence rate result in this case. We also show that our requirement that the objective function is strongly convex can be relaxed, and that our results in fact hold for a broader class of functions. In Section \ref{sec:lowerbound} we introduce a lower-bound for conditional gradient-based methods, that shows that for certain problems with a sparse optimal solution, our method is nearly optimal. Finally, in Section \ref{sec:experiments} we present empirical evidence which demonstrates the performance of our method.

\section{Preliminaries}\label{sec:prelim}

\begin{definition}
We say that a function $f(x):\reals^{n}\rightarrow\reals$ is $\alpha$-strongly convex w.r.t. a norm $\Vert\cdot\Vert$,
if for all $x,y\in\reals^{n}$ it holds that
\begin{eqnarray*}
f(y) \geq f(x) + \nabla{}f(x)\cdot(y-x) + \frac{\alpha}{2}\Vert{x-y}\Vert^2 .
\end{eqnarray*}
\end{definition}

\begin{definition}
We say that a function $f(x):\reals^{n}\rightarrow\reals$ is $\beta$-smooth w.r.t. a norm $\Vert\cdot\Vert$, if for all $x,y\in\reals^{n}$ it holds that
\begin{eqnarray*}
f(y) \leq f(x) + \nabla{}f(x)\cdot(y-x) + \frac{\beta}{2}\Vert{x-y}\Vert^2 .
\end{eqnarray*}
\end{definition}

The first-order optimality condition implies that for a $\alpha$-strongly convex $f$, if $x^*$ is the unique minimizer of $f$ over a convex and compact set $\mK\subset\reals^{n}$, then for all $x\in\mK$ it holds that
\begin{eqnarray}\label{eq:strongconvexdist}
f(x) - f(x^*) \geq \frac{\alpha}{2}\Vert{x-x^*}\Vert^2 .
\end{eqnarray}

Throughout this work we let $\Vert\cdot\Vert$ denote the standard euclidean norm.
Given a point $x\in\reals^n$, we let $\card(x)$ denote the number of non-zero entries in $x$.

\subsection{Setting}\label{sec:setting}
In this work we consider the following optimization problem:
\begin{eqnarray*}
\min_{x\in\mP}f(x).
\end{eqnarray*}
We make the following assumptions on $f$ and $\mP$:
\begin{itemize}
\item
$f(x)$ is $\alpha$-strongly convex and $\beta$-smooth with respect to the $\ell_2$ norm.

\item
$\mP$ is a polytope which satisfies the following two properties:
\begin{enumerate}
\item
$\mP$ can be described algebraically as $\mP = \{x\in\reals^n \, | \, x\geq 0, \, Ax = b\}$ . 
\item
All vertices of $\mP$ lie on the hypercube $\{0,1\}^n$.
\end{enumerate}
\end{itemize}

We let $x^*$ denote the (unique) minimizer of $f$ over $\mP$, and we let $D$ denote the Euclidean diameter of $\mP$, namely, $D=\max_{x,y\in\mP} \|x-y\|$. We let $\mV$ denote the set of vertices of $\mP$, where according to our assumptions, it holds that $\mV\subset\{0,1\}^n$.

%While the polytopes that fall into the above assumptions are far from being arbitrary, these assumptions already capture several important concepts such as paths in graphs, perfect-matchings, hidden Markov models, and more. We give detailed examples in Section \ref{sec:polyExample} in the appendix.
While the polytopes that satisfy the above assumptions are not completely general, these assumptions already capture several important concepts such as paths in graphs, perfect-matchings, Markov random fields, and more.
Indeed, a surprisingly large number of applications from machine learning, signal processing and other domains formulate optimization problems in this category (e.g., \cite{Jaggi13b, joulin2014fw, lacoste2015linear_fw}).
We give detailed examples of such polytopes in Section \ref{sec:polyExample}.
%Moreover, 
Importantly,
the above assumptions allow us to get rid of the dependency of the convergence rate on certain geometric parameters (such as $\psi,\xi$ in \cite{GH13} or the \textit{pyramidal width} in \cite{Jaggi13c,lacoste2015linear_fw}), which can be polynomial in the dimension, and hence result in an impractical convergence rate. %Finally, for many of the polytopes that %fall into
%satisfy
%the above assumptions 
Finally, for many of these polytopes, %it holds that 
the vertices are sparse, i.e., for any vertex $v\in\mV$, $\card(v) << n$. In this case, when the optimum $x^*$ can be decomposed as a convex combination of only a few vertices (and thus, sparse by itself), we get a sharper convergence rate that depends on the sparsity of $x^*$ and not explicitly on the dimension, as in previous works. 

We believe that our theoretical guarantees could be well extended to more general polytopes and we leave this extension for future work.

\section{Our Approach}\label{sec:approach}
In order to better communicate our ideas, we begin by first briefly introducing the standard conditional gradient method and its accelerated away-steps-based variants. We discuss both the blessings and shortcomings of these away-steps-based variants in Subsection \ref{sec:oldCG}. Then, in Subsection \ref{sec:newCG}, we present our new method, a decomposition-invariant away-steps-based conditional gradient algorithm, and discuss how it addresses the major shortcomings of previous away-steps-based variants.

\subsection{The conditional gradient method and acceleration via away-steps}\label{sec:oldCG}

The standard conditional gradient algorithm is given below (Algorithm \ref{alg:cg}). It is well known that when setting the step-size $\eta_t$ in an appropriate way, the worst case convergence rate of the method is $O(\beta{}D^2/t)$ \cite{Jaggi13b}. This convergence rate is tight for the method in general, see for instance  \cite{Lan13}.

\begin{algorithm}
\caption{Conditional Gradient}
\label{alg:cg}
\begin{algorithmic}[1]
\STATE Let $x_1$ be some vertex in $\mV$
\FOR{$t = 1...$}
\STATE $v_t \gets \arg\min_{v\in\mV}v\cdot\nabla{}f(x_t)$
\STATE choose a step-size $\eta_t\in(0,1]$
\STATE $x_{t+1} \gets (1-\eta_t)x_t + \eta_tv_t$
\ENDFOR
\end{algorithmic}
\end{algorithm}

Consider the iterate of Algorithm \ref{alg:cg} on iteration $t$, and let $x_t = \sum_{i=1}^k\lambda_iv_i$ be its convex decomposition into vertices of the polytope $\mP$. Note that Algorithm \ref{alg:cg}, implicitly discounts each coefficient $\lambda_i$ by a factor $(1-\eta_t)$, in favor of the new added vertex $v_t$.
A different approach, is not to decrease all vertices in the decomposition of $x_t$ uniformly, but to more-aggressively decrease vertices that are worse than others, with respect to some computable measure, such as their product with the gradient direction. This key principle proves to be crucial to breaking the $1/t$ rate of the standard method, and to achieve a linear convergence rate under certain strong-convexity assumptions, as described in the recent works \cite{GH13, lacoste2015linear_fw, Beck15}. For instance, in \cite{GH13} it was shown, via the introduction of the concept of a \textit{Local Linear Optimization Oracle}, that using such a non-uniform reweighing rule, in fact approximates a certain \textit{proximal} problem, that together with the shrinking effect of strong convexity, as captured by Eq. \eqref{eq:strongconvexdist}, yields a linear convergence rate. We refer to these methods as away-step-based CG methods.
As a concrete example, which will also serve as a basis for our new method, we bring the \textit{pairwise} variant recently studied in \cite{lacoste2015linear_fw}, which applies this principle in Algorithm \ref{alg:cga}, given below \footnote{While the convergence rate of this pairwise variant, established in \cite{lacoste2015linear_fw}, despite being linear, is significantly worse than other away-step-based variants, here we show on the contrary, that a proper analysis yields state-of-the-art performance guarantees.}. Note that Algorithm \ref{alg:cga} decreases the weight of exactly one vertex in the decomposition: that with the largest product with the gradient.

It is important to note that since previous away-step-based CG, unlike the original CG method, do not decrease the coefficients in the convex decomposition of the current iterate uniformly, they all require to explicitly store and maintain a convex decomposition of the current iterate. This issue raises two main disadvantages:

\paragraph{Superlinear memory and running-time overheads} Storing a decomposition of the current iterate as a convex combination of vertices of the polytope generally requires $O(n^2)$ memory. While the away-step-based variants increase the size of the decomposition by at most a single vertex per iteration, they also typically exhibit linear convergence after performing at least $O(n)$ steps \cite{GH13, lacoste2015linear_fw, Beck15}, and thus, this $O(n^2)$ estimate still holds. Moreover, since these methods require i) to find the worse vertex in the decomposition, in terms of dot-product with current gradient direction, and ii) to update this decomposition on each iteration (even when using sophisticated update techniques such as in \cite{Beck15}), the per-iteration over-head in terms of computation time of these methods is also at least $O(n^2)$.

\paragraph{Decomposition-specific performance}
While the choice of new vertex to be added in Algorithms \ref{alg:cg} is independent of a specific representation of the current iterate $x_t$ as a convex combination of vertices of the polytope, the choice of away-step in Algorithm \ref{alg:cga} does depend on the specific decomposition that is maintained by the algorithm. Since the feasible point $x_t$ may admit several different convex decompositions, committing to one such decomposition, might result in sub-optimal away-steps. Ideally, the away-steps, much like the standard CG methods, will be independent of any specific decomposition.
As observable in Table \ref{table:compare}, for certain problems in which the optimal solution is sparse, all analyses of previous away-steps-based variants are significantly suboptimal, since they all depend explicitly on the dimension, which seems to be an unavoidable side-effect of being decomposition-dependent. On the other hand, the fact that our new approach is decomposition-invariant allows us to obtain sharper convergence rates for such instances.

\begin{algorithm}
\caption{Pairwise Conditional Gradient}
\label{alg:cga}
\begin{algorithmic}[1]
\STATE Let $x_1$ be some vertex in $\mV$
\FOR{$t = 1...$}
\STATE let $\sum_{i=1}^{k_t}a_t^{(i)}v_t^{(i)}$ be an \textbf{explicitly maintained} convex decomposition of $x_t$
\STATE $v_t^+ \gets \arg\min_{v\in\mV}v\cdot\nabla{}f(x_t)$
\STATE $j_t \gets \arg\min_{j\in[k_t]}v_t^{(j)}\cdot(-\nabla{}f(x_t))$
%\STATE $v_t^{(j)} \gets \arg\min_{v\in\{v_t^{(1)},...,v_t^{(k_t)}\}}v\cdot(-\nabla{}f(x_t))$
%\IF{$[-\nabla{}f(x_t)\cdot(v_t^+ - x_t)] \leq [-\nabla{}f(x_t)\cdot(x_t - v_t^{(j)})]$}
%\IF{$[-\nabla{}f(x_t)\cdot(v_t^+ - x_t)] \leq [-\nabla{}f(x_t)\cdot(x_t - v_t^{(j_t)})]$}
%\STATE choose a step-size $\eta_t\in(0,1]$
%\STATE $x_{t+1} \gets x_t+ \eta_t(v_t^+-x_t)$
%\ELSE
%\STATE choose a step-size $\eta_t\in\left({0,\, a_t^{(j)}/(1-a_t^{(j)})}\right]$
%\STATE choose a step-size $\eta_t\in\left({0,\, a_t^{(j_t)}/(1-a_t^{(j_t)})}\right]$
%\STATE $x_{t+1} \gets x_t + \eta_t(x_t-v_t^{(j_t)})$
%\STATE $x_{t+1} \gets x_t + \eta_t(x_t-v_t^{(j)})$
%\ENDIF
\STATE choose a step-size $\eta_t\in(0,\, a_t^{(j_t)}]$
\STATE $x_{t+1} \gets x_t + \eta_t(v_t^+-v_t^{(j_t)})$
\STATE update the convex decomposition of $x_{t+1}$
\ENDFOR
\end{algorithmic}
\end{algorithm}

\subsection{A new decomposition-invariant pairwise conditional gradient method}\label{sec:newCG}

Our main observation is that in many cases of interest, given a feasible iterate $x_t$, one can in-fact compute an optimal away-step from $x_t$ without relying on any single specific decomposition. This observation allows us to overcome both of the main disadvantages of previous away-step-based CG variants. Our algorithm, which we refer to as a \textit{decomposition-invariant pairwise conditional gradient} (DICG), is given below in Algorithm \ref{alg:newCG}. 

\begin{algorithm}
\caption{Decomposition-invariant Pairwise Conditional Gradient}
\label{alg:newCG}
\begin{algorithmic}[1]
\STATE input: sequence of step-sizes $\{\eta_t\}_{t\geq 1}$
\STATE let $x_0$ be an arbitrary point in $\mP$
\STATE $x_1 \gets \arg\min_{v\in\mV}v\cdot\nabla{}f(x_0)$
\FOR{$t = 1...$}
\STATE $v^{+}_t \gets \arg\min_{v\in\mV}v\cdot\nabla{}f(x_t)$
%\STATE $v^{-}_t \gets \arg\max_{v\in\mV: \, \supp(v) \leq \supp(x_t)}v\cdot\nabla{}f(x_t)$
\STATE define the vector $\tilde{\nabla}f(x_t)\in\reals^m$ as follows:
\begin{eqnarray*}
\tilde{\nabla}f(x_t)_i := \left\{ \begin{array}{ll}
         \nabla{}f(x_t)_i & \mbox{if $x_t > 0$}\\
        -\infty & \mbox{if $x_t = 0$}\end{array} \right.
\end{eqnarray*}
\STATE $v^{-}_t \gets \arg\min_{v\in\mV}v\cdot\left({-\tilde{\nabla}f(x_t)}\right)$
\STATE choose a new step-size $\tilde{\eta}_t$ using one of the following two options:
\begin{description}
\item[Option 1: predefined step-size] \hfill \\
let $\delta_t$ be the smallest natural such that $2^{-\delta_t} \leq \eta_t$, and set a new step-size $\tilde{\eta}_t \gets 2^{-\delta_t}$
\item[Option 2: line-search] \hfill \\
%\begin{eqnarray*}
$\gamma_t \gets \max_{\gamma\in[0,1]}\{x_t+\gamma{}(v_t^+ - v_t^-)\geq 0\}, \quad \tilde{\eta}_t \gets \min_{\eta\in(0,\gamma_t]}f(x_t + \eta(v_t^+ - v_t^-))$
%\end{eqnarray*}
\end{description}
%\STATE let $\delta_t$ be the smallest natural such that $2^{-\delta_t} \leq \eta_t$, and set a new step-size $\tilde{\eta}_t \gets 2^{-\delta_t}$
\STATE $x_{t+1} \gets x_t + \tilde{\eta}_t(v^{+}_t - v^{-}_t)$
\ENDFOR
\end{algorithmic}
\end{algorithm}

The following observation details the optimality of away-steps taken by Algorithm \ref{alg:newCG}.

\begin{observation}[optimal away-steps in Algorithm \ref{alg:newCG}]\label{obsrv:optVminus}
Consider an iteration $t$ of Algorithm \ref{alg:newCG} and suppose that the iterate $x_t$ is feasible. Let $x_t=\sum_{i=1}^k\lambda_iv_i$ for some integer $k$, be an irreducible way of writing $x_t$ as a convex sum of vertices of $\mP$, i.e., $\lambda_i >0$ for all $i\in[k]$. Then it holds that
\begin{eqnarray*}
\forall i\in[k]: \quad v_i\cdot \nabla{}f(x_t) \leq v_t^- \cdot \nabla{}f(x_t), \qquad \gamma_t \geq \min\{x_t(i) \, | \, i\in[n], \, x_t(i) > 0\}.
\end{eqnarray*}
\end{observation}
\begin{proof}
Let $x_t=\sum_{i=1}^k\lambda_iv_i$ be a convex decomposition of $x_t$ into vertices of $\mP$, for some integer $k$, where each $\lambda_i$ is positive. Note that it must hold that for any $j\in[n]$ and any $i\in[k]$, $x_t(j) = 0 \Rightarrow v_i(j) = 0$, since by our assumption on $\mP$, $\mV\subset\reals^n_+$. 
The observation then follows directly from the definition of $v_t^-$.
\end{proof}

The following theorem which details the convergence rate of Algorithm \ref{alg:newCG} is the main theorem of this paper.

\begin{theorem}\label{thm:main}
Let $M_1 = \sqrt{\frac{\alpha}{8\card(x^*)}}$ and $M_2 = \frac{\beta{}D^2}{2}$. Consider running Algorithm \ref{alg:newCG} with Option 1 for the step-size, and 
suppose that 
\begin{eqnarray*}
\forall t\geq 1: \qquad \eta_t =  \frac{M_1}{2\sqrt{M_2}}\left({1-\frac{M_1^2}{4M_2}}\right)^{\frac{t-1}{2}}. 
\end{eqnarray*}
Then the iterates of Algorithm \ref{alg:newCG} are always feasible and satisfy: 
\begin{eqnarray*}
\forall t\geq 1: \qquad f(x_t) - f(x^*) \leq \frac{\beta{}D^2}{2}\exp\left({-\frac{\alpha}{8\beta{}D^2\card(x^*)}t}\right).
\end{eqnarray*}
\end{theorem}

The following corollary of Theorem \ref{thm:main} shows that the so-called \textit{duality gap}, defined as $g_t := (x_t - v^+_t)\cdot\nabla f(x_t)$, which serve as a certificate of the sub-optimality of the iterates of Algorithm \ref{alg:newCG}, also converges with a linear rate.

\begin{corollary}\label{corr:main}
For any iteration $t$ of Algorithm \ref{alg:newCG}, define the dual gap $g_t := (x_t - v_t^+)\cdot\nabla{}f(x_t)$, and observe that, since $f(x)$ is convex, $h_t \leq g_t$. Then, for any $t$ which satisfies: $h_t \leq \frac{\beta{}D^2}{2}$, it holds that
\begin{eqnarray*}
g_t \leq \sqrt{2\beta{}D^2h_t} .
\end{eqnarray*} 
\end{corollary}

We now turn to make several remarks regarding Algorithm \ref{alg:newCG} and Theorem \ref{thm:main}:
\begin{itemize}
\item Note that on any iteration $t$ of the algorithm, aside from the computation of the gradient vector $\nabla{}f(x_t)$ and the two calls to the linear optimization oracle of $\mP$, all other computations, when using the first option for choosing the step-size, can be carried out in $O(n)$ time and space. This is much more efficient than previous linearly convergent CG variant, such as those in \cite{GH13, lacoste2015linear_fw, Beck15}, which typically require at least additional $O(n^2)$ time and space per iteration, since they require to maintain an explicit convex decomposition of the iterates.

\item
Note that despite the different parameters of the problem at hand (e.g., $\alpha, \beta, D, \card(x^*)$), running the algorithm with Option 1 for choosing the step-size, for which the guarantee of Theorem \ref{thm:main} holds, requires the knowledge of a \textit{single} parameter, i.e., $M_1/\sqrt{M_2}$. In particular, it is an easy consequence that running the algorithm with an estimate $M\in[0.5M_1/\sqrt{M_2},\, M_1/\sqrt{M_2}]$, will only affect the leading constant in the convergence rate listed in the theorem. Hence, $M_1/\sqrt{M_2}$ could be efficiently estimated via a logarithmic-scale search.

\item
Theorem \ref{thm:main} improves significantly over the convergence rate established for the pairwise conditional gradient variant in 
\cite{lacoste2015linear_fw}. In particular, the number of iterations to reach an $\epsilon$ error in the analysis of \cite{lacoste2015linear_fw} depends linearly on $\vert{\mV}\vert!$, where $\vert{\mV}\vert$ is the number of vertices of $\mP$.
\end{itemize}

\section{Examples of Polytopes}\label{sec:polyExample}
% !TeX root = paper.tex
In this section we turn to survey several important examples of structured polytopes that fit the assumptions detailed in Subsection \ref{sec:setting} and detail the application of Algorithm \ref{alg:newCG} to optimization over these polytopes.

\paragraph*{Unit Simplex}
The simplex in $\reals^n$ is the set of all distributions over $n$ elements, i.e. the set:
\begin{eqnarray*}
\mS_n = \lbrace{x\in\mathbb{R}^n \, |\, x \geq 0 \, , \, \sum_{i=1}^nx_i =1}\rbrace .
\end{eqnarray*}
Alternatively, $\mS_n$ is the convex hull of all standard basis vectors in $\reals^n$.
 
It is easy to verify that $D = \sqrt{2}$.

Linear minimization over the simplex is trivial and can be carried out by a single pass over the non-zero elements in the linear objective. In particular, computing $v_t^-$ in Algorithm \ref{alg:newCG} simply amounts to finding the largest (signed) entry in $\nabla{}f(x_t)$ which corresponds to a non-zero entry in $x_t$, and thus is even more efficient than computing the standard CG direction $v_t^+$.

%\paragraph*{$L_1$-ball}
%Consider the Lasso problem:
%\[
%\min_x \|Hx - d\|^2 \qquad \text{s.t. } \|x\|_1\le k
%\]
%We can make the following substitution of variables to transform the problem into the desired form $\mP$. \\
%Denote by $z^+$ and $z^-$ the positive and the symmetric of the negative parts of $x$, respectively. Namely, $x = z^+ - z^-$, and for each coordinate, either $z^+_i$ or $z^-_i$ is zero.
%In that case: $\|x\|_1 = 1^\top (z^+ + z^-)$.
%To get vertices in $\{0,1\}^n$ rather than $\{0,k\}^n$ we need to further scale the variables by $1/k$.
%So finally we obtain the equivalent problem:
%\begin{align*}
%\min_{z^+,z^-,\beta} \|Hz^+ - Hz^- - d\|^2 \qquad \text{s.t. } z\ge0, \beta\ge0, \quad 1^\top z^+ + 1^\top z^- + \beta = k
%& \min_{\tilde{z}} \|\tilde{H}\tilde{z} - d\|^2 \qquad \text{s.t. } \tilde{z}\ge0, \;\; 1^\top\tilde{z} = 1 \\
%&\text{where } \tilde{z} = \left( \begin{array}{c} z^+ \\ z^- \\ \beta \end{array}\right),
%\qquad \tilde{H} = \left[ H , -H , \vec{0} \right]
%\end{align*}
%This has the desired polytope form (same as the Simplex case).\\
%\todo{need to show that indeed either $z^+_i$ or $z^-_i$ is zero and the two problems are equivalent.}

\paragraph*{Flow polytope}
Let $G$ be a \textit{directed acyclic graph} (DAG) with a set of vertices $V$ such that $\vert{V}\vert=n$, and a set of edges $E$ such that $\vert{E}\vert = m$, and let $s,t$ be two vertices in $V$ which we refer to as the \textit{source} and the \textit{target}, respectively. The $s-t$ flow polytope, denoted here by $\mF_{st}$, is the set of all unit $s-t$ flows in $G$, where for each point $x\in\mF_{st}$ and $i\in[m]$, the entry $x_i$ is the amount of flow through edge $i$ according to the flow $x$. $\mF_{st}$ is also known as the \textit{$s-t$ path polytope} since it is the convex hull of all identifying vectors of paths from $s$ to $t$ in the graph $G$. It is easy to verify that that $D < \sqrt{2n}$.

Since $\mF_{st}$ is the convex hull of paths, linear minimization is straightforward: given a linear objective $c\in\reals^m$, we need to find the identifying vector of the lightest $s-t$ path in $G$ with respect to the edge weights induced by $c$. Since the graph $G$ is a DAG, this could be carried out in $O(m)$ time \cite{schrijver}. In particular, computing the direction $v_t^-$ in Algorithm \ref{alg:newCG} over the flow polytope, amounts to finding the lightest $s-t$ path in $G$ with respect to the gradient vector $\nabla{}f(x_t)$, under the constraint that all edges on the path are assigned non-zero flow by $x_t$. Thus,  we can compute $v_t^-$ by running a shortest $s-t$ path algorithm after removing all edges with zero flow from the graph. Thus, as in the simplex case, computing $v_t^-$ is even more efficient than computing the standard direction $v_t^+$.

It is also important to note that when $G$ is not extremely sparse, i.e., when $m=\omega(n)$, it holds for every vertex $v$ of $\mF_{st}$ that $\card(v) << m$. Thus, if $x^*$ can be expressed as a combination of only a few paths, i.e., it corresponds to a sparse flow, it holds that $\card(x^*)$ is much smaller than the standard dimension of the problem $m$.

\paragraph*{Perfect Matchings polytope}
Let $G$ be a \textit{bipartite} graph with $n$ vertices on each side and $m$ crossing edges. The perfect matching polytope, denoted here by $\mM$, is the convex hull of all identifying vectors of perfect matchings in $G$. In case the two sides of $G$ are fully connected, this polytope is also known as the Birkhoff polytope - the set of all $n\times n$ \text{doubly stochastic matrices}, i.e. matrices with non-negative real entries whose entries along any row and any column add up to $1$. 
It easily follows that $D \leq \sqrt{2n}$.

In order to minimize a linear objective over $\mM$, we need to find a minimum-weight perfect matching in a bipartite graph, where the edge weights are induced by the linear objective. This could be carried out via combinatorial algorithms in $\min\{\tilde{O}(\sqrt{n}m),O(n^3)\}$ time \cite{schrijver}. As in the flow polyope, in this case also, computing $v_t^-$ is even more efficient than computing $v_t^+$, since it amounts to finding a minimum weight perfect matching after all edges that are zero-valued in $x_t$ are removed from the graph.

As in the flow polytope, in case $G$ is not trivially sparse, i.e., when $m=\omega(n)$, it holds that if $x^*$ could be expressed as a combination of only a few matchings in $G$, then $\card(x^*) << m$, where $m$ is the dimension of the problem.

\paragraph*{Marginal polytope}
In \emph{Graphical Models} several optimization problems are defined for variables representing marginal distributions over subsets of model variables.
There exists a set of linear constraints, known as the \emph{marginal polytope}, which guarantees that these variables are legal marginals of some global distribution \cite{jordan}.
For example, the learning problem in Max-Margin Markov Networks is defined as a quadratic program over the marginal polytope \cite{taskar03}.
%Specifically, linear prediction is formalized as solving: $y(x;w) = \argmax_{y\in{\cal{Y}}} w^\top\phi(x,y)$.
%It is often assumed that the output space has an internal structure. For example, the score function may take the form of a pairwise Markov Random Field (MRF): $w^\top\phi(x,y) = \sum_{i\in V(G)} w_i^\top\phi_i(x,y_i) + \sum_{ij\in E(G)} w_{ij}^\top\phi_{ij}(x,y_i,y_j)$.
%For general factors we have: $w^\top\phi(x,y) = \sum_{c} w_c^\top\phi_i(x,y_c) $, where $y_c$ is a subset of $y_i$'s.

For general graphical models the marginal polytope consists of an exponential number of constraints.
Fortunately, for some models, such as tree-structured graphs, the polytope can be characterized by a polynomial number of local consistency constraints, known as the \emph{local marginal polytope} \cite{jordan}.
Consider a set of discrete variables $(y_1,\ldots,y_n)$, and denote by $\mu_c(y_c)$ the marginal probability of an assignment to a subset of these variables $y_c$. Then the local marginal polytope is defined as:
%
%Given a training set of samples $\{(x^{(m)},y^{(m)})\}_{m=1}^M$, using a structured hinge loss and an $L_2$ regularizer, the ERM training objective becomes:
%\[
%\min_w \frac{\lambda}{2} \|w\|^2 + \frac{1}{M} \sum_m \max_y \left[ w^\top\hspace{-3pt}\left(\phi(x^{(m)},y)-\phi(x^{(m)}, y^{(m)})\right) \hspace{-2pt} + \Delta(y,y^{(m)}) \right]~,
%\]
%where $\Delta(y(x^{(m)};w),y^{(m)})$ is a task-specific loss measuring the quality of prediction.
%
%Under the MRF structure, the dual training objective is given by \cite{taskar03}:
%\[
%\min_{\mu\in{\cal{M}}^\times} \frac{\lambda}{2}\|\Psi\mu\|^2 - \mu^\top\ell~,
%\]
%where $\mu$ is the set of dual variables ($\mu^{(m)}_c(y_c)$ is the dual variable for sample $m$, factor $c$ and local assignment $y_c$), $\Psi_{m,c,y_c} = \frac{1}{\lambda M} \left(\phi_c(x^{(m)},y^{(m)}_c) - \phi_c(x^{(m)},y_c)\right)$ is a column vector in $\reals^d$,
%and $\ell_{m,c,y_c} = \frac{1}{M}\Delta(y_c,y^{(m)}_c)$ is a scalar (we assume that the task-loss decomposes as the model score).
%
%The set $\cal{M}$ is known as the \emph{marginal polytope} \cite{jordan}. It could in principle be written in the desired form $\mP$, however this would require an exponential number of linear constraints in $A$.
%Instead, in some problems, such as tree-structured graph, a polynomial number of constraints suffices to define the polytope. This set, denoted ${\cal{M}}_L$, is known as the \emph{local marginal polytope} \cite{jordan}:
\[
{\cal{M}}_L = 
\left\{
\mu\geq 0 : 
\begin{array}{ll}
\sum_{y_{c\setminus i}} \mu_c(y_c) = \mu_{i}(y_i) & \forall c,i\in c, y_i \\
\sum_{y_i} \mu_i(y_i) = 1 & \forall i
\end{array}
\right\}
\]
%More generally, this can be used as an approximate training objective:
%\[
%\min_{\mu\in{\cal{M}}_L^\times} \frac{\lambda}{2}\|\Psi\mu\|^2 - \mu^\top\ell
%\]
For tree-structured graphs ${\cal{M}}_L$ is known to have only integral vertices \cite{jordan}, so it has the desired form assumed in Section \ref{sec:setting}.

In this case $D= \sqrt{2} |{\cal{C}}|$, where ${\cal{C}}$ is the number of subsets $y_c$ (factors in the graphical model).

In many interesting cases linear optimization over the marginal polytope can be implemented efficiently via dynamic programming. For example, for chain-structured graphs the Viterbi algorithm is used.
Finally, we note that computing the direction $v_t^-$ in Algorithm \ref{alg:newCG} can often be cheaper than computing $v_t^+$, since the restriction to the support of $x_t$ can eliminate many of the possible configurations of marginals.

\section{Analysis}\label{sec:analysis}
In this section we turn to analyze the performance of Algorithm \ref{alg:newCG}, and prove Theorem \ref{thm:main}.

Throughout this section we let $h_t$ denote the approximation error of Algorithm \ref{alg:newCG} on iteration $t$, for any $t\geq 1$, i.e., $h_t = f(x_t) - f(x^*)$.

\subsection{Feasibility of the iterates generated by Algorithm \ref{alg:newCG}}

We start by proving that the iterates of Algorithm \ref{alg:newCG} are always feasible. While feasibility is straightforward when using the the line-search option to set the step-size (Option 2), it is less obvious when using the first option.

\begin{observation}\label{observ:feasbileStep}
Suppose that on some iteration $t$ of Algorithm \ref{alg:newCG}, the iterate $x_t$ is feasible, and that the step-size is chosen using Option 1. Then, if for all $i\in[n]$ for which $x_t(i) \neq 0$ it holds that $x_t(i) \geq \tilde{\eta}_t$, then the following iterate $x_{t+1}$ is also feasible.
\end{observation}
\begin{proof}
From the optimality of $v_t^-$ it follows that for any $i\in[n]$, if $x_t(i) = 0$, then $v_t^-(i) = 0$ (note in particular that any vertex with positive weight in some convex decomposition of $x_t$ must satisfy this condition). Thus, from our assumption on the size of positive entries in $x_t$, and since $v_t^-\in\{0,1\}^n$, it follows that the vector $w_t := x_t - \tilde{\eta}_tv_t^-$, satisfies: $w_t \geq 0$. Since $v_t^+$ is feasible it also follows that $x_{t+1} = w_t + \tilde{\eta}_t \geq 0$. Finally, since $x_t, v_t^-, v_t^+$ are all feasible, it also holds that $Ax_{t+1} = b$. Thus, $x_{t+1}$ is feasible.
\end{proof}

\begin{lemma}[feasibility of iterates under Option 1]\label{lem:feasible}
Suppose that the sequence of step-sizes $\{\eta_t\}_{t\geq 1}$ is monotonically non-increasing, and contained in the interval $[0,1]$. Then,
the iterates generated by Algorithm \ref{alg:newCG} using Option 1 for setting the step-size, are always feasible.
\end{lemma}
\begin{proof}
We are going to prove by induction that on each iteration $t$ there exists a non-negative integer-valued vector $s_t\in\mathbb{N}^n$, such that
for any $i\in[n]$, it holds that $x_t(i) = 2^{-\delta_t} s_t(i)$. The lemma then follows by applying Observation \ref{observ:feasbileStep}, and since by definition, $\tilde{\eta}_t = 2^{-\delta_t}$.

The base case $t=1$ holds since $x_1$ is a vertex of $\mP$ and thus for any $i\in[n]$ we have that $x_1(i)\in\{0,1\}$ (recall that $\mV\subset\{0,1\}^n$). On the other hand, since $\eta_1 \leq 1$, it follows that $\delta_1 \geq 0$. Thus, there indeed exists a non-negative integer-valued vector $s_1$, such that $x_1 = 2^{-\delta_1} s_1$.

Suppose now that the induction holds for some $t\geq 1$. Since by definition of $v_t^-$, subtracting $\tilde{\eta}_tv_t^-$ from $x_t$ can only decrease positive entries in $x_t$ (see proof of Observation \ref{observ:feasbileStep}), and both $v_t^-, v_t^+$ are vertices of $\mP$ (and thus in $\{0,1\}^n$), and $\tilde{\eta}_t = 2^{-\delta_t}$, it follows that each entry $i$ in $x_{t+1}$ is given by: 
{\small \begin{eqnarray*}
x_{t+1}(i) = 2^{-\delta_t}\left\{ \begin{array}{ll}
         s_t(i) & \mbox{if $s_t(i) \geq 1$ \& $v_t^-(i) = v_t^+(i) = 1$ or $v_t^-(i) = v_t^+(i) = 0$}\\
         s_t(i)-1 & \mbox{if $s_t(i) \geq 1$ \& $v_t^-(i) = 1$ \& $v_t^+(i) = 0$}\\
         s_t(i)+1 & \mbox{if $v_t^-(i) = 0$ \& $v_t^+(i) = 1$}\end{array} \right.
\end{eqnarray*}}
Thus, $x_{t+1}$ can also be written in the form $2^{-\delta_t}\tilde{s}_{t+1}$ for some $\tilde{s}_{t+1}\in\mathbb{N}^n$. By definition of $\delta_t$ and the monotonicity of $\{\eta_t\}_{t\geq 1}$, we have that  $\frac{2^{-\delta_t}}{2^{-\delta_{t+1}}}$ is a positive integer. Thus, setting $s_{t+1} = \frac{2^{-\delta_t}}{2^{-\delta_{t+1}}}\tilde{s}_{t+1}$, the induction holds also for $t+1$.
\end{proof}
\subsection{Bounding the per-iteration error-reduction of Algorithm \ref{alg:newCG}}
The following technical lemma is the key to deriving the linear convergence rate of our method, and in particular, to deriving the improved dependence on the sparsity of $x^*$, instead of the dimension. At a high-level, the lemma translates the $\ell_2$ distance between two feasible points into a $\ell_1$ distance in a simplex defined over the set of vertices of the polytope, which as we will show, is a natural way to measure distances for conditional gradient-based methods.
\begin{lemma}\label{lem:l1dist}
Let $x,y\in\mP$. There exists a way to write $x$ as a convex combination of vertices of $\mP$, $x=\sum_{i=1}^k\lambda_iv_i$ for some integer $k$, such that $y$ can be written as $y = \sum_{i=1}^k(\lambda_i-\Delta_i)v_i + (\sum_{i=1}^k\Delta_i)z$ with $\Delta_i\in[0,\lambda_i] \, \forall i\in[k]$,$z\in\mP$, and
%\begin{eqnarray*}
$\sum_{i=1}^k\Delta_i \leq \sqrt{\card(y)}\Vert{x-y}\Vert $.
%\end{eqnarray*}
\end{lemma}

\begin{proof}
Consider writing $y$ as some convex combination of vertices, $y=\sum_{i=1}^s\gamma_iu_i$ for some appropriate integer $s$. Applying Lemma 5.3. from \cite{GH13}, it follows that we can write $x$ as
\begin{eqnarray}\label{eq:lem:l1dist:1}
x = \sum_{i=1}^s(\gamma_i-\tilde{\Delta}_i)u_i + (\sum_{i=1}^s\tilde{\Delta}_i)\tilde{z},
\end{eqnarray}
where $\tilde{\Delta}_i\in[0,\gamma_i] \, \forall i\in[s]$, $\tilde{z}\in\mP$, and for every $i$ with $\tilde{\Delta}_i > 0$ there exists $j_i\in[n]$ such that $\tilde{z}(j_i) =0$ and $u_i(j_i) > 0$. Since each $u_i$ is a vertex of $\mP$ and thus a point of the $\{0,1\}$-hypercube, it further follows that $u_i(j_i) = 1$.
 Let $C = \{j_i \, | \, i\in[s]\}$. Observe that $\vert{C}\vert \leq \card(y)$. Now, we have that
\begin{eqnarray*}
\Vert{x-y}\Vert^2 &=& \Vert{\sum_{i=1}^s\tilde{\Delta}_i(u_i-\tilde{z})}\Vert^2 \geq \sum_{j\in{}C}\left({\sum_{i=1}^s\tilde{\Delta}_i(u_i(j)-\tilde{z}(j))}\right)^2
= \sum_{j\in{}C}\left({\sum_{i=1}^s\tilde{\Delta}_iu_i(j)}\right)^2  \\
&\geq & \frac{1}{\vert{C}\vert}\left({\sum_{j\in{}C}\sum_{i=1}^s\tilde{\Delta}_iu_i(j)}\right)^2 \geq \frac{1}{\vert{C}\vert}\left({\sum_{i=1}^s\tilde{\Delta}_i}\right)^2 .
\end{eqnarray*}
Rearranging we have that 
\begin{eqnarray}\label{eq:leml1dist:2}
\sum_{i=1}^s\tilde{\Delta}_i \leq \sqrt{\vert{C}\vert}\Vert{x-y}\Vert \leq \sqrt{\card(y)}\Vert{x-y}\Vert.
\end{eqnarray}

Note that using the convex decomposition of $x$ as in Eq. \eqref{eq:lem:l1dist:1}, and the bound in Eq. \eqref{eq:leml1dist:2} it follows that we can rewrite $y$ as a convex decomposition as suggested in the lemma.
\end{proof}

\begin{lemma}\label{lem:errReduction}
Consider the iterates of Algorithm \ref{alg:newCG}, when the step-sizes are chosen using Option 1.
Let $M_1 = \sqrt{\frac{\alpha}{8\card(x^*)}}$ and $M_2 = \frac{\beta{}D^2}{2}$. For any $t\geq 1$ it holds that
\begin{eqnarray*}
h_{t+1} \leq h_t - \eta_tM_1h_t^{1/2}  +\eta_t^2M_2.
\end{eqnarray*}
\end{lemma}
\begin{proof}
Define $\Delta_t = \sqrt{\frac{2\card(x^*)h_t}{\alpha}}$, and note that from Eq. \eqref{eq:strongconvexdist} we have that $\Delta_t \geq \sqrt{\card(x^*)}\Vert{x_t-x^*}\Vert$. 

As a first step, we are going to show that the point $y_t := x_t + \Delta_t(v_t^{+} - v_t^{-})$ satisfies: $y_t\cdot\nabla{}f(x_t) \leq x^*\cdot\nabla{}f(x_t)$.

From Lemma \ref{lem:l1dist} it follows that we can write $x$ as a convex combination $x_t = \sum_{i=1}^k\lambda_iv_i$ and write $x^*$ as $x^* = \sum_{i=1}^k(\lambda_i-\Delta_i)v_i + \sum_{i=1}^k\Delta_iz$, where $\Delta_i\in[0,\lambda_i]$, $z\in\mP$, and $\sum_{i=1}^k\Delta_i \leq \Delta_t$. %From Observation \ref{obsrv:optVminus} it follows, that for all $i\in[k]$, $v_t^-\cdot\nabla{}f(x_t) \geq v_i\cdot\nabla{}f(x_t)$. Also, by definition of $v_t^+$ it follows that $v_t^+\cdot\nabla{}f(x_t) \leq z\cdot\nabla{}f(x_t)$. Thus, it follows that
It holds that
\begin{eqnarray*}
(y_t - x_t) \cdot \nabla{}f(x_t) &=& \Delta_t(v_t^+ - v_t^-)\cdot\nabla{}f(x_t) \leq \sum_{i=1}^k\Delta_i(v_t^+ - v_t^-) \cdot\nabla{}f(x_t)\\
&\leq & \sum_{i=1}^k\Delta_i(z - v_i) \cdot\nabla{}f(x_t) = (x^*-x_t)\cdot\nabla{}f(x_t) ,
\end{eqnarray*}
where the first inequality follows since $(v_t^+ - v_t^-)\cdot\nabla{}f(x_t) \leq 0$, and the second inequality follows from the optimality of $v_t^+$ and $v_t^-$ (Observation \ref{obsrv:optVminus}). Rearranging, we have that indeed

\begin{eqnarray}\label{eq:lem:errReduction:1}
\left({x_t + \Delta_t(v_t^{+} - v_t^{-})}\right)\cdot\nabla{}f(x_t) \leq x^*\cdot\nabla{}f(x_t) ,
\end{eqnarray}
as needed.

Observe now that from the definition of $\tilde{\eta}_t$ it follows for any $t\geq 1$ that $\frac{\eta_t}{2} \leq \tilde{\eta}_t \leq \eta_t$.
Using the smoothness of $f(x)$ we have that
\begin{eqnarray*}
h_{t+1} &=& f(x_t + \tilde{\eta}_t(v_t^{+} - v_t^{-})) - f(x^*)\\
& \leq & h_t + \tilde{\eta}_t(v_t^{+} - v_t^{-})\cdot\nabla{}f(x_t) + \frac{\tilde{\eta}_t^2\beta}{2}\Vert{v_t^{+} - v_t^{-}}\Vert^2 \\
& \leq & h_t + \tilde{\eta}_t(v_t^{+} - v_t^{-})\cdot\nabla{}f(x_t) + \frac{\tilde{\eta}_t^2\beta{}D^2}{2} \\
& \leq & h_t + \frac{\eta_t}{2}(v_t^{+} - v_t^{-})\cdot\nabla{}f(x_t) + \frac{\eta_t^2\beta{}D^2}{2} \\
&=& h_t + \frac{\eta_t}{2\Delta_t}\left({(x_t + \Delta_t(v_t^{+} - v_t^{-}) - x_t}\right)\cdot\nabla{}f(x_t) + \frac{\tilde{\eta}_t^2\beta{}D^2}{2} \\
&\leq & h_t + \frac{\eta_t}{2\Delta_t}\left({x^* - x_t}\right)\cdot\nabla{}f(x_t) + \frac{\eta_t^2\beta{}D^2}{2} \\
&\leq &h_t - \frac{\eta_t}{2\Delta_t}h_t + \frac{\eta_t^2\beta{}D^2}{2}\\
&= & h_t - \frac{\eta_t\sqrt{\alpha}}{2\sqrt{2\card(x^*)h_t}}h_t + \frac{\eta_t^2\beta{}D^2}{2},
\end{eqnarray*}
where the third inequality follows since $(v_t^{+} - v_t^{-})\cdot\nabla{}f(x_t) \leq 0$, the forth inequality follows from Eq. \eqref{eq:lem:errReduction:1}, the fifth inequality follows from convexity of $f(x)$, and the last equality follows from plugging the value of $\Delta_t$.
\end{proof}

\subsection{Proof of Theorem \ref{thm:main}}

We now turn to prove Theorem \ref{thm:main}. Afterwards, we prove Corollary \ref{corr:main}.
\begin{proof}
We are first going to prove the convergence rate stated in the theorem, assuming that all iterates are feasible. Then we will show that for our choice of step-sizes, indeed the iterates are feasible.
We are going to prove by induction that there exist $c_0,c_1$ such that for all $t\geq 1$ it holds that $h_t \leq c_0(1-c_1)^{t-1}$. Clearly for the base case we must require that $c_0 \geq h_1$.

Suppose now that the induction holds for some $t\geq 1$. Let us set 
\begin{eqnarray}\label{eq:mainthm:eta}
\eta_t = \frac{M_1}{2M_2}\sqrt{c_0}(1-c_1)^{\frac{t-1}{2}}.
\end{eqnarray}

Using Lemma \ref{lem:errReduction} and the induction hypothesis we have that

\begin{eqnarray*}
h_{t+1} &\leq & h_t - \frac{M_1^2}{2M_2}\sqrt{c_0(1-c_1)^{t-1}}h_t^{1/2} + \frac{M_1^2}{4M_2}c_0(1-c_1)^{t-1} \\
&\leq & h_t - \frac{M_1^2}{2M_2}h_t + \frac{M_1^2}{4M_2}c_0(1-c_1)^{t-1} \\
&= & h_t\left({1-  \frac{M_1^2}{2M_2}}\right) + \frac{M_1^2}{4M_2}c_0(1-c_1)^{t-1} \\
&\leq & c_0(1-c_1)^{t-1}\left({1 - \frac{M_1^2}{4M_2}}\right),
\end{eqnarray*}
where the induction hypothesis was used in both the second and third inequalities. In the third inequality we have also used the fact that
\begin{eqnarray}\label{eq:main:1}
\frac{M_1^2}{2M_2} = \frac{\alpha}{8\beta\card(x^*)D^2} < 1,
\end{eqnarray}
where the inequality follows since $\alpha \leq \beta$ and both $\card(x^*),D$ are at least $1$.

Thus, if we set $c_1 = \frac{M_1^2}{4M_2}$, the induction follows. %Note that
%\begin{eqnarray}\label{eq:main:1}
%\frac{M_1^2}{4M_2} = \frac{\alpha}{16\beta\card(x^*)D^2} < 1,
%\end{eqnarray}
%where the inequality follows since $\alpha \leq \beta$. Thus, $c_1\in(0,1)$ as needed.

We now turn to figure out $c_0$.

Using the smoothness of $f(x)$ and the choice of $x_1$ in Algorithm \ref{alg:newCG}, we have that
\begin{eqnarray*}
h_1 &=& f(x_1) - f(x^*) = f(x_0 + (x_1-x_0)) - f(x^*) \\
&\leq &f(x_0) - f(x^*) + (x_1-x_0)\cdot\nabla{}f(x_0) + \frac{\beta\Vert{x_0-x_1}\Vert^2}{2} \\%\leq \frac{\beta{}D^2}{2} \\
&\leq &f(x_0) - f(x^*) + (x^*-x_0)\cdot\nabla{}f(x_0) + \frac{\beta\Vert{x_0-x_1}\Vert^2}{2} \leq \frac{\beta{}D^2}{2},
\end{eqnarray*}
where the last inequality follows from the convexity of $f(x)$.

Thus, we can set $c_0 = \frac{\beta{}D^2}{2} = M_2$, which completes the proof of the convergence rate.

Now, it remains to prove that indeed all iterates are feasible. First note that the sequence $\{\eta_t\}_{t\geq 1}$, as defined in Eq. \eqref{eq:mainthm:eta} is monotonically non-increasing. Furthermore, plugging the values $M_1,M_2,c_0$, we have that
\begin{eqnarray*}
\eta_1 =\frac{M_1\sqrt{c_0}}{2M_2} = \frac{1}{2}\sqrt{\frac{M_1^2}{M_2}} = \frac{1}{2}\sqrt{\frac{\alpha}{4\beta{}D^2\card(x^*)}} \leq 1,
\end{eqnarray*}
where the inequality follows similarly to the one in Eq. \eqref{eq:main:1}.
Thus, our choice of step-size sequence $\{\eta_t\}_{t\geq1}$ satisfies the conditions of Lemma \ref{lem:feasible}, and thus it follows that all iterates of Algorithm \ref{alg:newCG} are feasible.
\end{proof}

We now prove Corollary \ref{corr:main}.
\begin{proof}
Fix an iteration $t$. Using the $\beta$-smoothness of $f(x)$ we have that
\begin{eqnarray*}
\forall\eta\in(0,1]: \quad f(x^*) \leq f(x_t + \eta(v_t^+ - x_t)) \leq f(x_t) + \eta(v_t^+ - x_t)\cdot\nabla{}f(x_t) + \frac{\eta^2\beta{}D^2}{2} .
\end{eqnarray*}
Rearranging we have that
\begin{eqnarray*}
\forall\eta\in(0,1]: \quad g_t = (x_t-v_t^+)\cdot\nabla{}f(x_t) \leq \frac{1}{\eta}h_t + \frac{\eta\beta{}D^2}{2} .
\end{eqnarray*}
Thus, when $\sqrt{\frac{2h_t}{\beta{}D^2}} \leq 1$, we can set $\eta = \sqrt{\frac{2h_t}{\beta{}D^2}}$ in the above inequality, and obtain the corollary.
\end{proof}

\section{Extensions}\label{sec:extension}

In this section we detail two extensions of our result: i) relaxing the specific structure of the polytope $\mP$ considered in Subsection \ref{sec:setting}, and ii) relaxing the strong convexity requirement on the objective function $f(x)$.

\subsection{Extension of Algorithm \ref{alg:newCG} to arbitrary polytopes}

In this subsection we detail how to extend our approach to a broader class of polytopes. While proving rigorous guarantees for this extension is beyond the scope of this paper and left for future work, the encouraging experimental results for Algorithm \ref{alg:newCG} with line-search, suggest that this extended variant, for which line-search is also possible, may also exhibit favorable empirical performance. 
Towards this end, in this subsection we consider minimizing a smooth and strongly-convex function over an arbitrary polytope $\mP$ which we assume is given in the following way:
\begin{eqnarray*}
\mP = \{x\in\reals^n \, | \, A_1x = b_1, \, A_2x \leq b_2\},
\end{eqnarray*}
where $A_2$ is $m\times n$. We assume that given a point $x\in\reals^n$, we have an efficient way to evaluate the vector $A_2x$, which is indeed the case for most structured polytopes of interest.

\begin{algorithm}[H]
\caption{Decomposition-invariant Pairwise Conditional Gradient with Line-search for Arbitrary Polytopes}
\label{alg:newCGext}
\begin{algorithmic}[1]
\STATE let $x_0$ be an arbitrary point in $\mP$
\STATE $x_1 \gets \arg\min_{v\in\mV}v\cdot\nabla{}f(x_0)$
\FOR{$t = 1...$}
\STATE $v^{+}_t \gets \arg\min_{v\in\mV}v\cdot\nabla{}f(x_t)$
%\STATE $v^{-}_t \gets \arg\max_{v\in\mV: \, \supp(v) \leq \supp(x_t)}v\cdot\nabla{}f(x_t)$
\STATE define the vector $c\in\reals^m$ as follows:
\begin{eqnarray*}
c_i := \left\{ \begin{array}{ll}
         0 & \mbox{if $A_2(i)\cdot x_t < b_2(i)$}\\
        \infty & \mbox{if $A_2(i)\cdot x_t = b_2(i)$}\end{array} \right.
\end{eqnarray*}
\STATE $v^{-}_t \gets \arg\min_{v\in\mV}\left({-\nabla{}f(x_t)}\right)\cdot{}v + c^{\top}A_2v$
\STATE $\gamma_t \gets \max\{\gamma \in[0,1] \, | \, A_2(x_t+\gamma(v_t^+-v^{-}_t)) \leq b_2\}$
\STATE $\eta_t \gets \arg\min_{\eta\in[0,\gamma_t]}f(x_t + \eta(v_t^+ - v_t^-))$
\STATE $x_{t+1} \gets x_t + \eta_t(v^{+}_t - v^{-}_t)$
\ENDFOR
\end{algorithmic}
\end{algorithm}

\begin{observation}[optimal away-step for an arbitrary polytope]\label{observ:generalizedAwayStep}
Consider an iteration $t$ of Algorithm \ref{alg:newCGext} and suppose that the iterate $x_t$ is feasible. Let $x_t=\sum_{i=1}^k\lambda_iv_i$ for some integer $k$, be a irreducible way of writing $x_t$ as a convex sum of vertices of $\mP$, i.e., $\lambda_i >0$ for all $i\in[k]$. Then it holds that
\begin{eqnarray*}
\forall i\in[k]: \quad v_i\cdot \nabla{}f(x_t) \leq v_t^- \cdot \nabla{}f(x_t), \qquad \gamma_t > 0 . 
\end{eqnarray*}
Moreover, there exists a convex decomposition of $x_t$ that assigns a weight at least $\gamma_t$ to $v_t^-$.
\end{observation}
\begin{proof}
Let $x_t = \sum_{i=1}^k\lambda_iv_i$ be a decomposition of $x_t$ into vertices of $\mP$ such that $\lambda_i > 0$ for all $i\in[k]$. Observe that for any $j\in[m]$ and $i\in[k]$ it holds that $A_2(j)\cdot x_t = b_2(j) \Rightarrow A_2(j)\cdot v_i = b_2(j)$.
Note that by definition of the vector $c$ and $v_t^-$ it holds that
\begin{eqnarray}\label{eq:extend:1}
v_t^- &\in & \arg\max_{v\in\mV}\nabla{}f(x_t)\cdot v - c^{\top}A_2v \equiv \arg\max_{v\in\mV}\nabla{}f(x_t)\cdot v + c^{\top}(b_2 - A_2v) \nonumber \\
&\equiv& {\arg\max}_{v\in\{y\in\mV \, | \, \forall j\in[m]: \, A_2(j)\cdot x_t = b_2(j) \Rightarrow A_2(j)\cdot y = b_2(j)\}}v\cdot\nabla{}f(x_t) .
\end{eqnarray}
Thus, it follows that for all $i\in[k]$, $v_t^-\cdot\nabla{}f(x_t) \geq v_i\cdot\nabla{}f(x_t)$.

In order to prove the second part of the observation, we note that from the RHS of Eq. \eqref{eq:extend:1} it follows that there exists $\gamma_t > 0$ such that indeed $x_t - \gamma_tv_t^- \leq (1-\gamma_t)b_2$. To see this, consider some $j\in[m]$. If $A_2(j)\cdot{}x_t = b_2(j)$, then from the RHS of Eq. \eqref{eq:extend:1}, it follows that $A_2(j)\cdot{}v_t^- = b_2(j)$ and thus, for any $\gamma_t$ it holds that $A_2(j)\cdot(x_t - \gamma_tv_t^-) = (1-\gamma_t)b_2(j)$. Otherwise, there exists some $\epsilon_j > 0$ such that $A_2(j)\cdot{}v_t^- \leq b_2(j) - \epsilon_j$. Thus, for small enough, yet positive $\gamma_t$ we will have that $A_2(j)\cdot(x_t - \gamma_tv_t^-) \leq (1-\gamma_t)b_2(j)$.
Since it clearly also holds that $A_1(x_t - \gamma_tv_t^-) = (1-\gamma_t)b_1$, we have that the vector $w_t:=x_t - \gamma_tv_t^-$ satisfies: $w_t\in(1-\gamma_t)\mP$. Hence, $w_t$ can be decomposed as $w_t = \sum_{i=1}^q\tilde{\gamma}_i\tilde{v}_i$, where $q$ is a positive integer and for all $i\in[q]$, $\tilde{\lambda}_i > 0$, $\tilde{v}_i$ is a vertex of $\mP$, and $\sum_{i=1}^q\tilde{\lambda}_i = 1-\gamma_t$. Thus, since $v_t^-$ is a vertex of $\mP$, it follows that $x_t = w_t + \gamma_tv_t^-$ admits the convex decomposition $\sum_{i=1}^q\tilde{\lambda}_i\tilde{v}_i + \gamma_tv_t^-$, as needed.
\end{proof}

The following lemma is an immediate consequence of the choice of $\gamma_t$ in  Algorithm \ref{alg:newCGext}.
\begin{lemma}
The iterates of Algorithm \ref{alg:newCGext} are always feasible.
\end{lemma}

\subsection{Relaxing the strong convexity of the objective function}
Until now we have assumed that the objective function $f$ is strongly convex. However, as can be observed from our analysis, the only consequence of strong convexity that we relied on in our analysis, is Eq. \eqref{eq:strongconvexdist}. Indeed, there exist functions which are not strongly convex, that under certain conditions, still satisfy Eq. \eqref{eq:strongconvexdist}, and thus are compatible with our method and analysis.

Following the work of Beck and Shtern \cite{Beck15}, we can consider a broader class of objective functions, namely functions that take the following form:
\begin{eqnarray}\label{eq:relaxFunc}
f(x) = g(Ax) + b\cdot x,
\end{eqnarray}
where $A\in\reals^{m\times n}$, and $g:\reals^m\rightarrow\reals$ is smooth and strongly convex.

In \cite{Beck15} (Lemma 2.5) it was shown, using an application of Hoffman's lemma, that there exists a constant $\kappa$ which depends both on the condition number of $g$ and the parameters $A,b$, such that for any feasible point $x$, it holds that
\begin{eqnarray}\label{eq:extStrongConvDist}
\min_{y\in{}\mP^*}\Vert{x-y}\Vert^2 \leq \kappa\left({f(x) - f^*}\right),
\end{eqnarray}
where $\mP^*\subset\mP$, is the set of all feasible points that minimize $f(x)$ over $\mP$, and $f^*$ is the minimum value of $f(x)$ over $\mP$.

It is easy to verify that Eq.~\eqref{eq:extStrongConvDist} can be readily used in our analysis instead of Eq.~\eqref{eq:strongconvexdist}, and thus our results extend to handle objectives of the form given in Eq. \eqref{eq:relaxFunc}.

We note that now, the dependency in our analysis and in Theorem \ref{thm:main} on the strong convexity parameter $\alpha$ will be replaced with $\kappa$, and the dependency on $\card(x^*)$ will be replaced with $\max_{y\in\mP^*}\card(y)$.

%\begin{lemma}[computing implicit away steps]
%Given a point $x\in\mP$ define a vector $c\in\reals^m$ as
%\begin{eqnarray*}
%c_i = \left\{ \begin{array}{ll}
%         0 & \mbox{if $A_2(i)\cdot x < b_2(i)$}\\
%        \infty& \mbox{else}\end{array} \right.
%\end{eqnarray*}
%Then, $v_t^{out}$ defined as:
%\begin{eqnarray}\label{eq:awayStep}
%y^*\in\arg\max_{y\in\mP}y\cdot\nabla{}f(x) - c^{\top}(b_2-A_2y),
%\end{eqnarray}
%is the optimal away-step from $x$.
%\end{lemma}

\section{Lower Bound for Problems with a Sparse Solution}\label{sec:lowerbound}

In this section we present a simple lower bound on the approximation error of, informally speaking, any natural conditional gradient variant that when initialized with a vertex of the feasible set, its iterate after $t$ iterations admits a convex decomposition into at most $t+1$ vertices of the polytope. That is, on each iteration, at most a single new vertex is added to the decomposition. The lower bound shows that there exists a $1$-smooth and $1$-strongly convex function $f$, for which, any such CG variant which is applied to the minimization of $f$ over the unit simplex, must take $\Omega(\card(x^*))$ steps before entering the linear convergence regime. To date, none of the previous analyses of linearly converging CG variants matches this lower bound since, in this exact setting, they all require, in worst-case, $\Omega(n)$ steps before entering the linear convergence regime, i.e., number of steps that is independent of $\card(x^*)$. 

To the best of our knowledge, Algorithm \ref{alg:newCG} and the corresponding Theorem \ref{thm:main} are the first to match this lower bound. We emphasize that the idea behind the construction of this lower bound is well known and follows almost immediately from previous constructions, such as those in \cite{Jaggi13b,GH13}.

\begin{lemma}
Fix an even integer $k\in[n]$, and consider the optimization problem
\begin{eqnarray*}
\min_{x\in\mS_n}\{f(x) := \frac{1}{2}\Vert{x - \frac{1}{k}\onesvec_k}\Vert^2\},
\end{eqnarray*}
where $\mS_n$ denotes the unit simplex in $\reals^n$, i.e., $\mS_n := \{x\in\reals^n \, | \, x\geq 0, \, \Vert{x}\Vert_1 =1\}$, and $\onesvec_k$ is a vector in $\reals^n$, defined as:
\begin{eqnarray*}
\onesvec_k(i) = \left\{ \begin{array}{ll}
         1 & \mbox{if $1\leq i \leq k$}\\
        0 & \mbox{else}\end{array} \right. 
\end{eqnarray*}
Observe that $x^* = \frac{1}{k}\onesvec_k$ is the unique minimizer of $f$ over $\mS_n$.
Then, any point $x\in\mS_n$, for which it holds that $\card(x) \leq k/2$ satisfies:
\begin{eqnarray*}
f(x) - f(x^*) \geq \frac{1}{4k} .
\end{eqnarray*}
\end{lemma}
\begin{proof}
Fix a point $x\in\mS_n$ for which it holds that $\card(x) \leq k/2$. In order to lower bound the approximation error of $x$, it suffices to consider the entries which are zero for $x$ and non-zero for $x^*$. Thus, we have that
\begin{eqnarray*}
f(x) \geq \frac{1}{2}\cdot\frac{k}{2}\cdot\left({0- \frac{1}{k}}\right)^2 = \frac{1}{4k}.
\end{eqnarray*}
\end{proof}

\section{Experiments}\label{sec:experiments}
% !TeX root = paper.tex
In this section we illustrate the performance of our algorithm in numerical experiments.
We use the two experimental settings from \cite{lacoste2015linear_fw}, which include a constrained Lasso problem and a video co-localization problem. In addition, we test our algorithm on a learning problem related to an optical character recognition (OCR) task from \cite{taskar03}.
In each setting we compare the performance of our algorithm (DICG) to standard conditional gradient (CG), as well as to the fast away (ACG) and pairwise (PCG) variants \cite{lacoste2015linear_fw}. For the baselines in the first two settings we use the publicly available code from \cite{lacoste2015linear_fw}, to which we add our own implementation of Algorithm \ref{alg:newCG}. Similarly, for the OCR problem we extend code from \cite{osokin16}, kindly provided by the authors.
For all algorithms we use line-search to set the step size.

\paragraph{Lasso}
In the first example the goal is to solve the problem: $\min_{x\in{\cal{M}}} \|\bar{A}x-\bar{b}\|^2$, where $\cal{M}$ is a scaled $\ell_1$ ball.
Notice that the constraints $\cal{M}$ do not match the required structure of $\mP$, however, with a simple change of variables we can obtain an equivalent optimization problem over the simplex.
We generate the random matrix $\bar{A}$ and vector $\bar{b}$ as in \cite{lacoste2015linear_fw}.
In Figure \ref{fig:all} (left, top) we observe that our algorithm (DICG) converges similarly to the pairwise variant PCG and faster than the other baselines. This is expected since the away direction $v^-$ in DICG (Algorithm \ref{alg:newCG}) is equivalent to the away direction in PCG (Algorithm \ref{alg:cga}) in the case of simplex constraints.
%Figure \ref{fig:lasso} (right) shows that our algorithm is actually faster than a naive implementation of the baselines.
%We point out that in contrast to these results, the convergence rate of the PCG %pairwise variant
%obtained in \cite{lacoste2015linear_fw} is significantly worse than the rate for ACG. %the away-steps variant.
%On the other hand, our theoretical analysis suggests that the pairwise variant is expected to perform well compared to the away-steps variant in this setting, which is better aligned with the empirical results.

\begin{figure*}[t]
%\vspace{-0.4cm}
 \begin{center}
  \begin{tabular}{c c c}
  Lasso & Video co-localization & OCR \\
  \hspace{-15px}
  \includegraphics[height=1.3in]{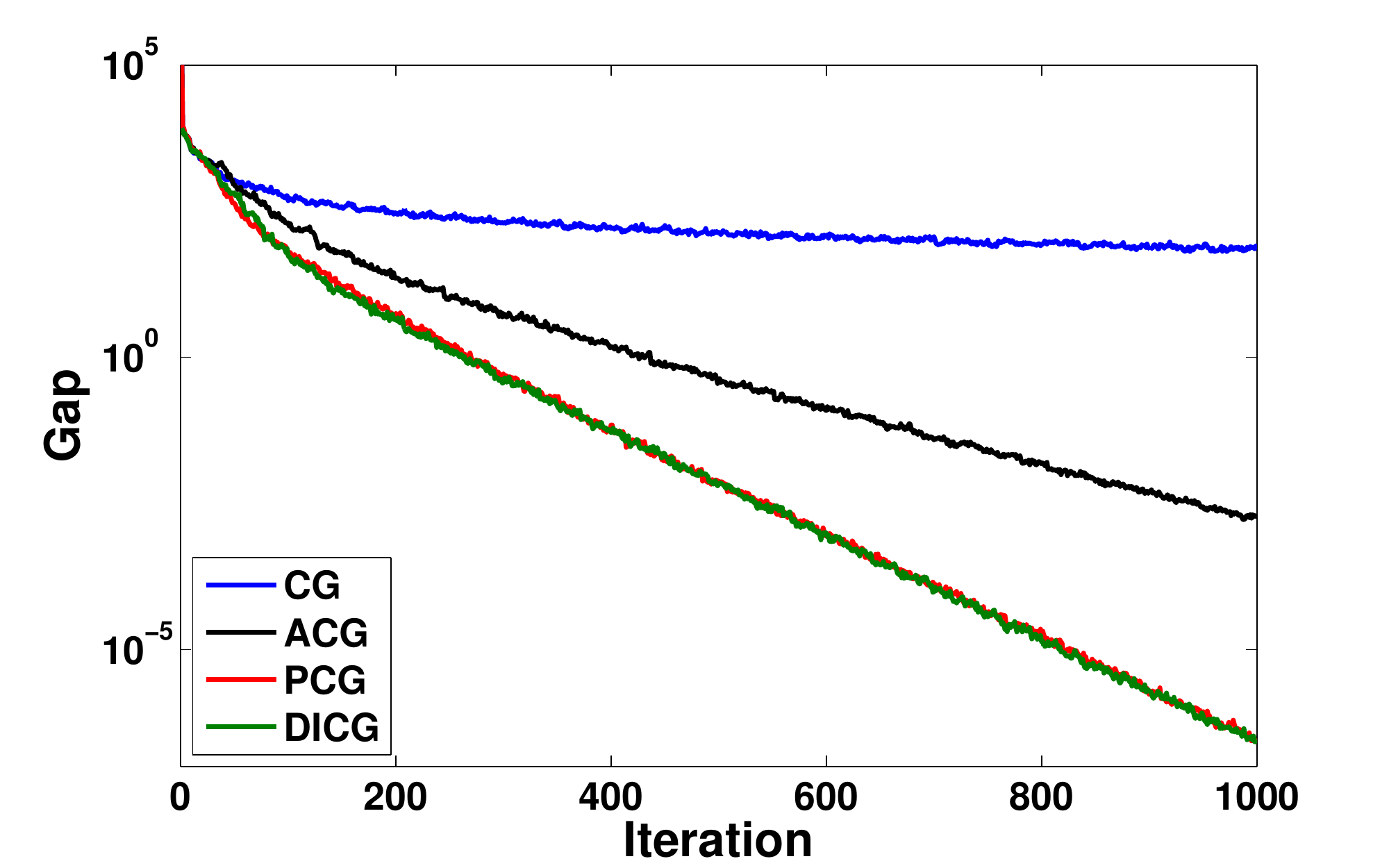}
  &
  \hspace{-20px}
  \includegraphics[height=1.3in]{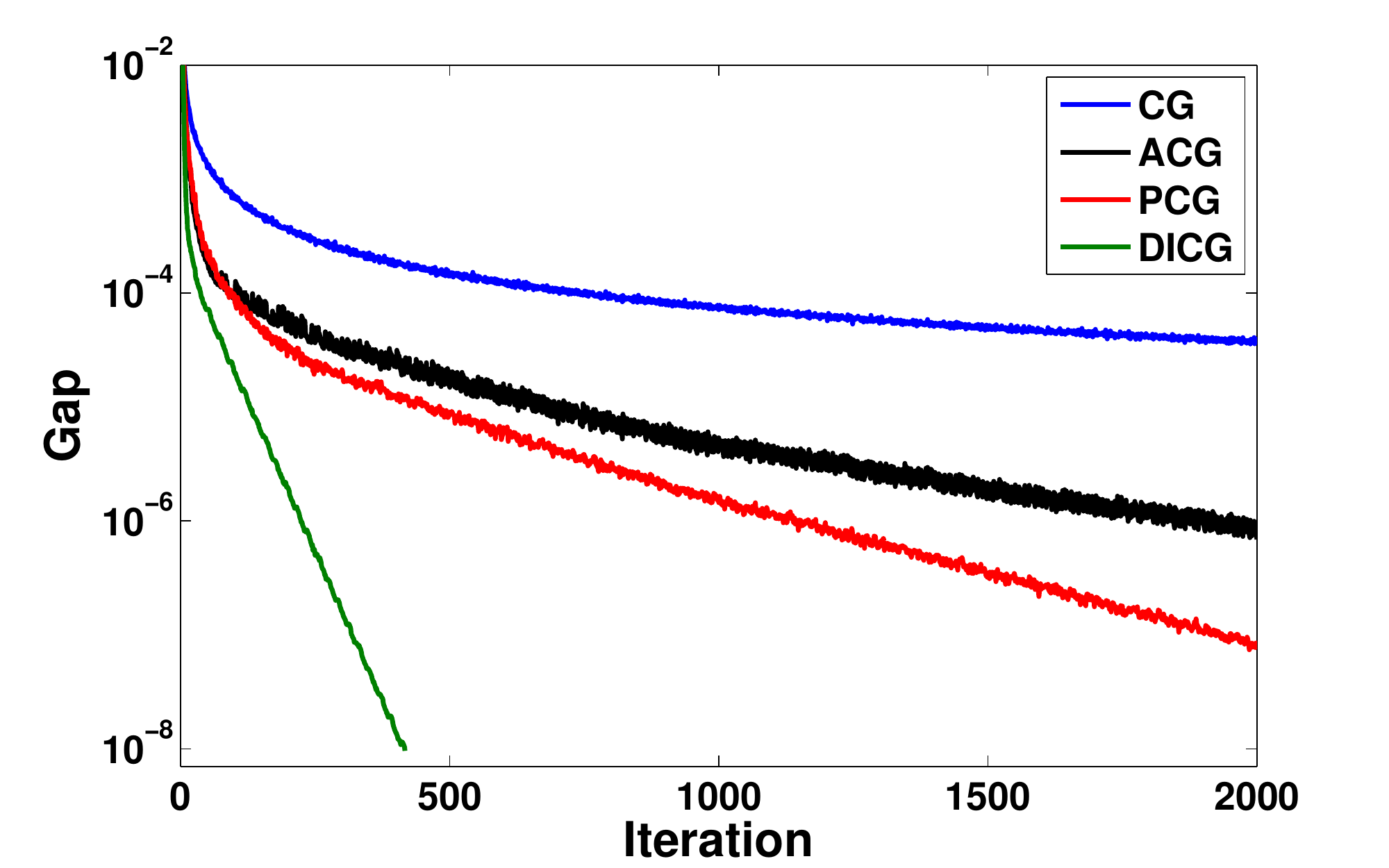}
  &
  \hspace{-20px}
  \includegraphics[height=1.3in]{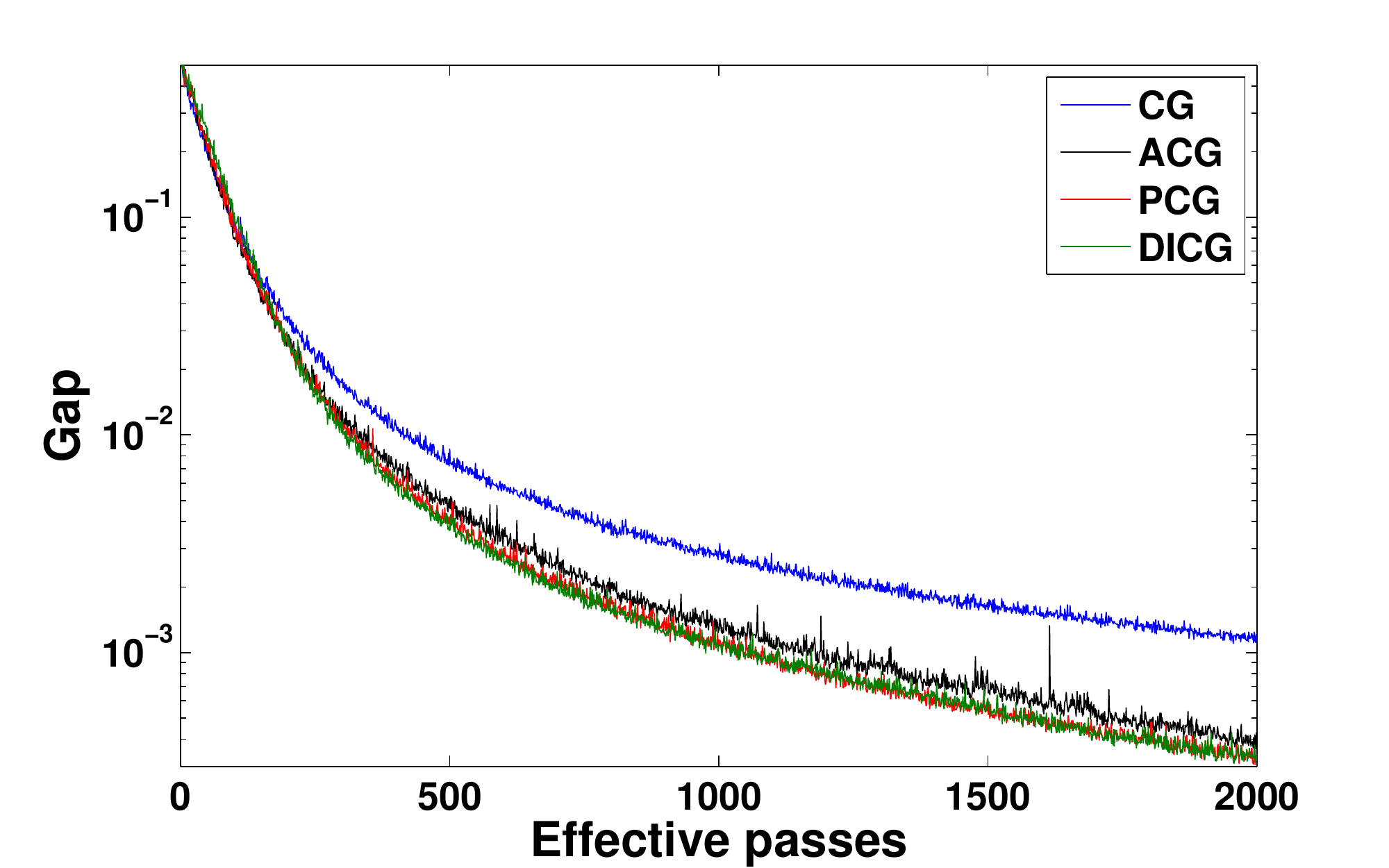}
  \\
  \hspace{-15px}
  \includegraphics[height=1.3in]{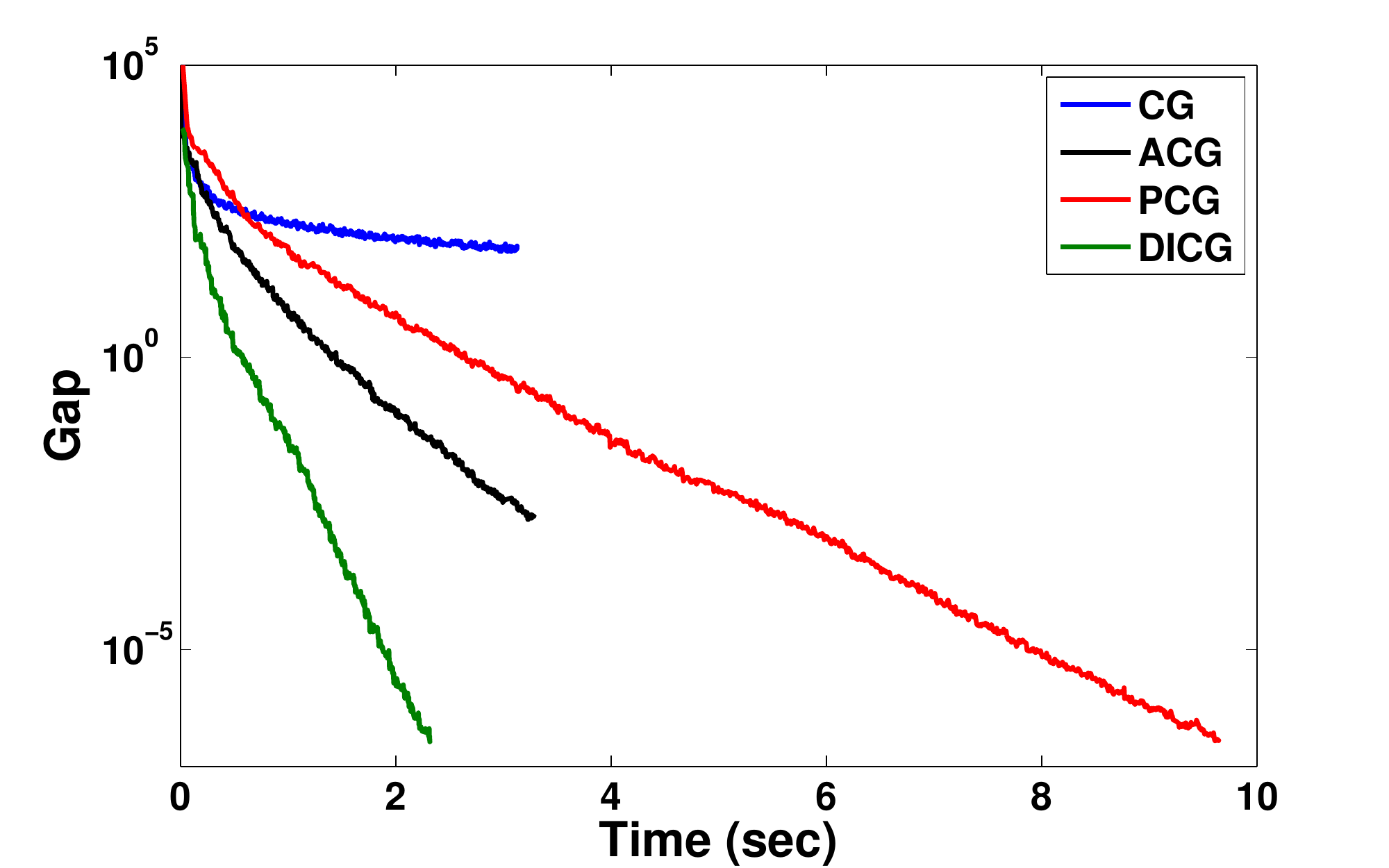}
  &
  \hspace{-20px}
  \includegraphics[height=1.3in]{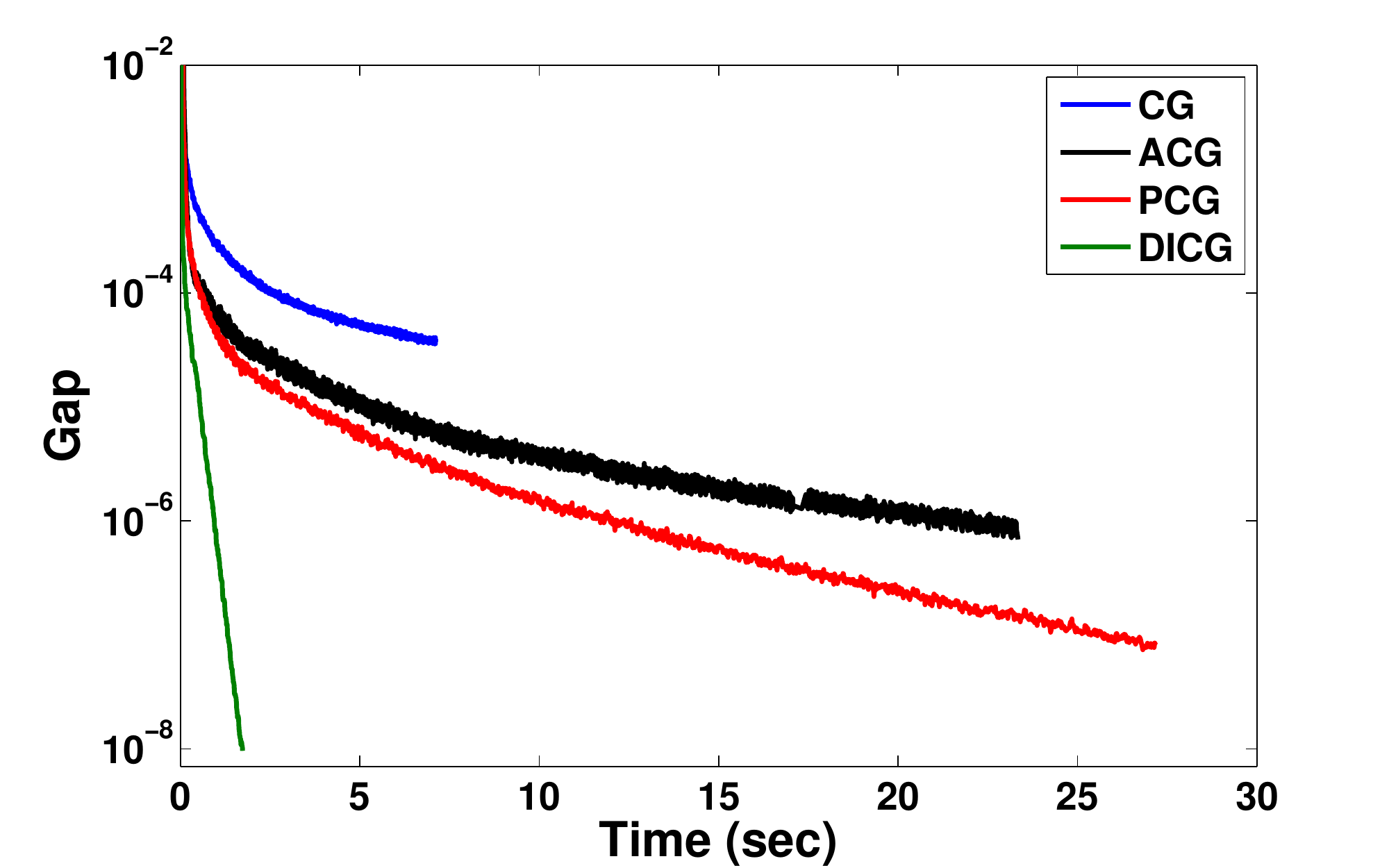}
  &
  \hspace{-20px}
  \includegraphics[height=1.3in]{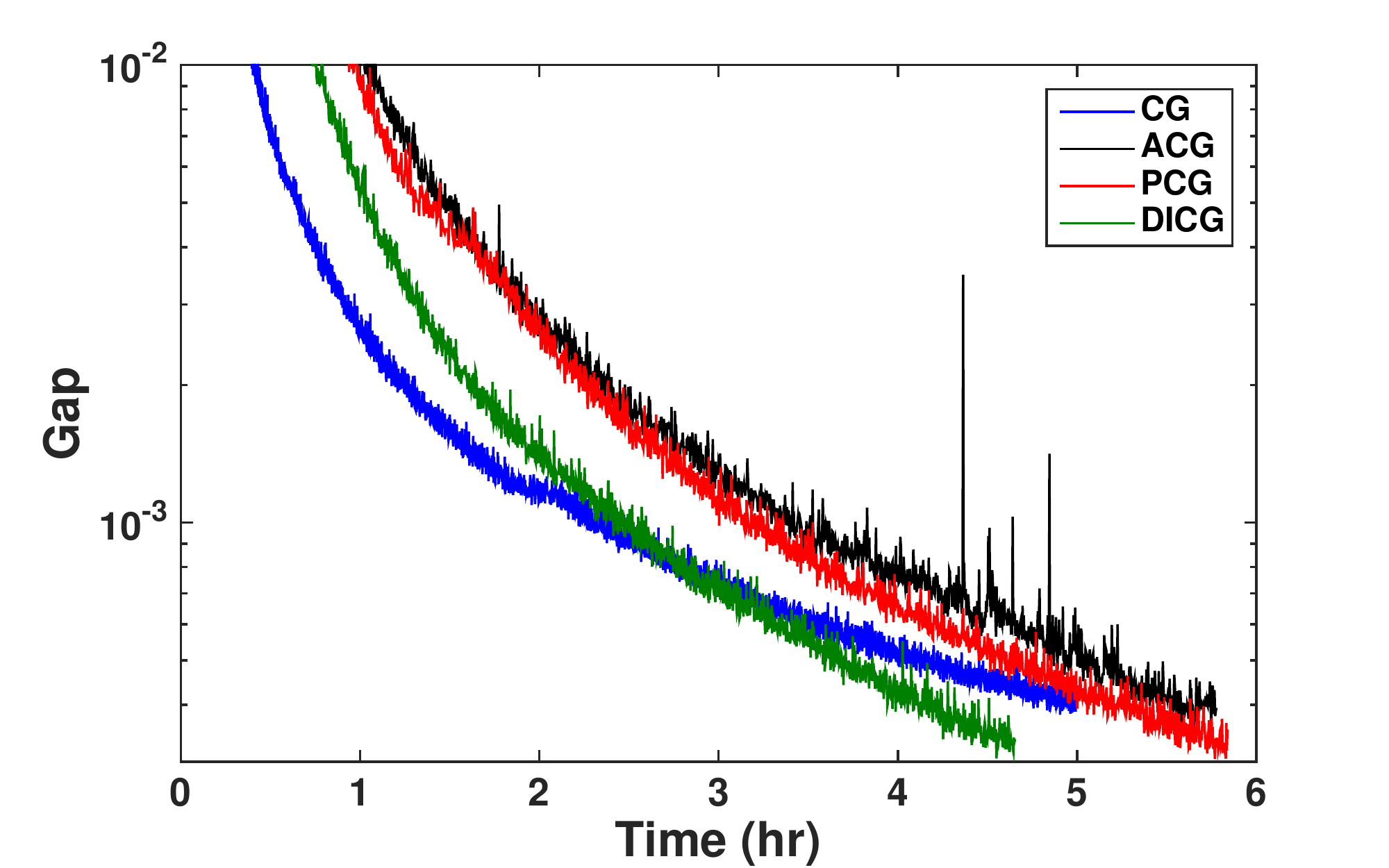}
  \end{tabular}
 \end{center}
\vspace{-4mm}
 \caption{{Duality gap $g_t$ vs.~iterations (top) and time (bottom) in various settings.}} %Lasso (left), video co-localization (middle), and OCR (right) problems.}}
  \label{fig:all}
\vspace{-0.4cm}
\end{figure*}

\paragraph{Video co-localization}
The second example is a quadratic program over the flow polytope, originally proposed in \cite{joulin2014fw}.
This is an instance of $\mP$ that is mentioned in Section \ref{sec:polyExample} in the appendix.
As can be seen in Figure \ref{fig:all} (middle, top), in this setting our proposed algorithm significantly outperforms the baselines, as a result of finding a better away direction $v^-$.
Figure \ref{fig:all} (middle, bottom) shows convergence on a time scale, where the difference between the algorithms is even larger.
One reason for this difference is the costly search over the history of vertices maintained by the baseline algorithms.
Specifically, the number of stored vertices grows fast with the number of iterations and reaches 1222 for away steps and 1438 for pairwise steps (out of 2000 iterations).

%\begin{figure*}[t]
%%\vspace{-0.4cm}
% \begin{center}
%  \begin{tabular}{c c}
%  \hspace{-10px}
%  \includegraphics[width=3in]{figs/coloc_NewPairFW_iter}
%  &
%  \hspace{-10px}
%  \includegraphics[width=3in]{figs/coloc_NewPairFW_time}
%  \end{tabular}
% \end{center}
%\vspace{-4mm}
% \caption{{Duality gap vs.~iterations (left) and time (right) on a video co-localization problem.}}
%  \label{fig:video}
%\vspace{-0.4cm}
%\end{figure*}

\paragraph{OCR}
We next conduct experiments on a structured SVM learning problem resulting from an OCR task.
The constraints in this setting are the marginal polytope corresponding to a chain graph over the letters of a word (see \cite{taskar03}), and the objective function is quadratic.
Notice that the marginal polytope has a concise characterization in this case and also satisfies our assumptions (see Section \ref{sec:polyExample} in the appendix for more details).
For this problem we actually run Algorithm \ref{alg:newCG} in a block-coordinate fashion, where blocks correspond to training examples in the dual SVM formulation \cite{Jaggi13a,osokin16}.
In Figure \ref{fig:all} (right, top) we see that our DICG algorithm is comparable to the PCG algorithm and faster than the other baselines on the iteration scale. Figure \ref{fig:all} (right, bottom) demonstrates that in terms of actual running time we get a noticeable speedup compared to all baselines. We point out that for this OCR problem, both ACG and PCG each require about 5GB of memory to store the explicit decomposition in the implementation that we used, so using DICG instead results in significant memory savings.

%\begin{figure*}[t]
%%\vspace{-0.4cm}
% \begin{center}
%  \begin{tabular}{c c}
%  \hspace{-10px}
%  \includegraphics[width=3in]{figs/ocr_iter}
%  &
%  \hspace{-10px}
%  \includegraphics[width=3in]{figs/ocr_time}
%  \end{tabular}
% \end{center}
%\vspace{-4mm}
% \caption{{Duality gap vs.~iterations (left) and time (right) on an OCR problem.}}
%  \label{fig:ocr}
%\vspace{-0.4cm}
%\end{figure*}

\bibliographystyle{plain}
\bibliography{bib}
\end{document}